\documentclass[11pt]{article}

\usepackage[english]{babel}
\usepackage{csquotes}
\usepackage{amssymb}
\usepackage{amsthm}
\usepackage{amsmath}
\usepackage{mathrsfs}
\usepackage[a4paper, margin=3cm]{geometry}
\usepackage{color}
\usepackage{hyperref}
\usepackage{numprint}
\usepackage{multirow, array}
\usepackage{placeins}
\usepackage{makecell}
\usepackage{enumerate}

\usepackage{url}

\theoremstyle{plain}
\newtheorem{theorem}{Theorem}[section]
\newtheorem{conjecture}[theorem]{Conjecture}
\newtheorem{lemma}[theorem]{Lemma}

\newtheorem{corollary}[theorem]{Corollary}
\newtheorem{algorithm}[theorem]{Algorithm}
\newtheorem*{mainconjecture}{Conjecture~\ref{conj:stability_squarefree}}
\newtheorem*{easyhalf}{Theorem~\ref{thm:stability_easy_half}}
\newtheorem{example}[theorem]{Example}

\theoremstyle{definition}

\theoremstyle{remark}

\numberwithin{equation}{section}

\newcommand{\CC}{\mathbb{C}}

\newcommand{\bS}{\mathbb{S}}
\newcommand{\QQ}{\mathbb{Q}}
\newcommand{\TT}{\mathbb{T}}
\newcommand{\ZZ}{\mathbb{Z}}

\newcommand{\SL}{\operatorname{SL}}
\newcommand{\skn}{S_k^{\mathrm{new}}(\Gamma_0(N))}
\def \new {\mathrm{new}}

\def \cS {\mathcal{S}}
\def \fS {\mathfrak{S}}
\def \lam {\lambda}
\def \cF {\mathcal{F}}
\def \cB {\mathcal{B}}

\title{Distinguishing newforms\footnote{Thanks to James Withers
for running some of the code. We also thank David
Loeffler, M. Ram Murty, Abhishek Saha and Fredrik Str\"omberg for useful comments.}}
\author{
Sam Chow\footnote{The first author was supported by the Elizabeth and Vernon
  Puzey scholarship, and is grateful towards the University of Melbourne for
their hospitality while preparing this memoir.}\\
University of Bristol\\
{\tt sam.chow42@gmail.com}\\
\and 
Alexandru Ghitza\footnote{The second author was supported by 
Discovery Grant DP120101942 from the Australian Research Council.}\\
University of Melbourne \\
{\tt aghitza@alum.mit.edu} \\
}
\date{\today}

\begin{document}
\thispagestyle{empty}

\maketitle

\begin{abstract}
Let $n_0(N,k)$ be the number of initial Fourier coefficients necessary to
distinguish newforms of level $N$ and even weight $k$. We produce extensive
data to support our conjecture that if $N$ is a fixed squarefree positive integer
and $k$ is large then $n_0(N,k)$ is the least prime that
does not divide $N$. 
\end{abstract}

\section{Introduction}

The predominant way of specifying a modular form is via its Fourier expansion
\begin{equation} \label{Fourier}
  f(q)=\sum_{n=0}^\infty a_n(f)q^n.
\end{equation}
Since this power series representation involves infinitely many coefficients,
a natural question is whether one can recognise a given form $f$ by looking
at only finitely many (initial) coefficients.  This is the object of a
classical result of Sturm:\footnote{Theorem~\ref{thm:sturm} is a special case of Sturm's main
result, which is concerned with congruences between modular forms.  The
particular case we state here was very likely already known to Hecke, long
before Sturm's work.}

\begin{theorem}[\cite{Stu1987}, see also~\cite{Mur1997}]
  \label{thm:sturm}
  Let $N, k\in\ZZ_{>0}$, and let $f, g\in M_k(\Gamma_0(N))$ with $f\neq g$.
  Then there exists
  \begin{equation}
    \label{eq:sturm}
    n\leq \frac{k}{12}\,[\SL_2(\ZZ)\colon \Gamma_0(N)]
  \end{equation}
  such that $a_n(f)\neq a_n(g)$.
\end{theorem}

We refer to~\eqref{eq:sturm} as the \emph{Sturm
bound}. It is sharp at this
level of generality.  However, many modular forms that occur naturally in
applications (especially in number-theoretic contexts) have additional
properties, such as being eigenvectors for the Hecke operators.  We ask the
question: \emph{Is it possible to sharpen the Sturm bound in the presence of
this extra information?} More precisely, let $n_0=n_0(N,k)$ be the smallest nonnegative
integer such that the following statement is true:

\begin{quotation}
  \noindent
  Let $f,g\in S_k(\Gamma_0(N))$ be newforms such that $a_n(f)=a_n(g)$ for all
  $n\leq n_0$.  Then $f=g$.
\end{quotation}

The main problem studied in this paper is the dependence of $n_0$ on the
parameters $N$ and $k$, for $N, k \in \ZZ_{>0}$.
Note that if $k$ is odd then $n_0 = 0$, since
$S_k(\Gamma_0(N)) = \{ 0 \}$ (see \cite[p. 15]{Kil2008}). We therefore
restrict attention to even weights $k$.
The empirical data that we computed strongly support the
following stability conjecture for squarefree levels:
\begin{mainconjecture}
  Let $N\in\ZZ_{>0}$ be squarefree. Then there exists $K\in\ZZ_{>0}$ such that if
  $k\geq K$ is an even integer then $n_0(N,k)$ is equal to the least prime that does 
not divide $N$.
\end{mainconjecture}

We note that the least prime that does 
not divide $N$ is bounded
above by $2(\log N + 1)$; see the proof of~\cite[Theorem~1]{Ghi2011}. The 
data also indicate a stability phenomenon in the non-squarefree level
case, but we have not found a simple conjectural characterisation of the
eventual value of $n_0$ in
this situation---there are cases where it appears to exceed the least prime that does 
not divide $N$. We can prove the ``easy half'' of
Conjecture~\ref{conj:stability_squarefree}:

\begin{easyhalf}
  Let $N \in \ZZ_{>0}$, and let $k \ge 38$ be an even integer. 
Then $n_0(N,k)$ is greater than or equal to the least prime that does 
not divide $N$.
\end{easyhalf}

Many authors have studied the recognition problem for modular forms, e.g.
\cite{BCP2002},
\cite{CG2014}, \cite{Ghi2011},
\cite{GS2014}, \cite{GH1993},
\cite{Koh2004}, 
\cite{Kow2005}, \cite{Mor1985},
\cite{Mur1997},
\cite{Stu1987}.
Maeda's conjecture (Conjecture \ref{conj:maeda}) would imply that $n_0(1,k) \le 2$
for all $k \in \ZZ_{>0}$, and from this and \cite[Theorem 1]{Mar2005}
it would follow
that $n_0(1,k) = 2$ for all even numbers $k \ge 28$. 
The second author and J. Withers \cite{GW2015} have proposed a
generalisation of Maeda's conjecture that would imply 
Conjecture \ref{conj:stability_squarefree}; this is discussed in \S \ref{gen_maeda}.

Our algorithm for evaluating $n_0(N,k)$ is based on the fact that if $f$ is a normalised
Hecke eigenform and $n \in \ZZ_{>0}$ then $a_n(f)$ equals the eigenvalue of 
$f$ with respect to the Hecke operator $T_n$. Moreover, it suffices to consider $T_p$
for primes $p$ (see Lemma \ref{primes}). We consider
intersections of $T_p$-eigenspaces over all primes $p$ up to a point (call 
such intersections ``homes'').
Any home $H$ has a basis given by its newforms, so the number of newforms
in $H$ is equal to $\dim H$. We continue until there are no homes of dimension
greater than one.

Two main refinements improve the efficiency of our algorithm. We use modular
symbols instead of modular forms (see \eqref{iso}). Our second improvement is 
harder to describe. The idea is to factorise over $\QQ$ the characteristic 
polynomial of $T_p$, considering the kernel of each
irreducible factor. We intersect these kernels for small primes $p$, and run our
algorithm on each such intersection. This enables us to
work in smaller spaces, reducing the need to manipulate large matrices.

This paper is organised as follows. In \S \ref{background}, we clarify definitions
and recall key results. In \S \ref{algorithm}, we describe in detail our algorithm
for computing $n_0(N,k)$. In \S \ref{stability}, we discuss 
Conjecture \ref{conj:stability_squarefree} in more detail. In particular, we prove
Theorem \ref{thm:stability_easy_half}, and further address the case where
$N$ is not squarefree. Finally, in 
\S \ref{gen_maeda}, we relate Conjecture \ref{conj:stability_squarefree}
to a conjecture of the second author and J. Withers.

For $k \in \ZZ_{>0}$ and $\Gamma$ a congruence subgroup, 
we denote by $M_k(\Gamma)$ the complex
vector space of weight $k$ modular forms for $\Gamma$, and
write $S_k(\Gamma)$ for its cuspidal subspace.
The symbol $p$ is reserved for primes. For $r \in \ZZ_{>0}$, we write
$p_r$ for the $r$th smallest prime number. We shall write
$\omega(N)$ for the number of distinct prime divisors of $N$. The
algebraic closure of $\QQ$ will be denoted $\overline{\QQ}$.

\section{Some background}
\label{background}

In this section we recall some standard definitions and results. Let $N,k \in
\ZZ_{>0}$. We shall work in
\begin{equation} \label{Sdef}
S := \skn,
\end{equation}
the new subspace of $S_k(\Gamma_0(N))$. Here we refer the reader to
\cite[\S I.6]{DI1995}. (The space $S$ was first defined in \cite{AL1970}.) 

For each $n \in \ZZ_{>0}$ we have a Hecke operator $T_n$ acting on $S$,
and these commute (see \cite[\S 5.3]{DS2005} and \cite[Ch. 9]{Ste2007}). 
The \emph{Hecke algebra} is the commutative ring
generated by the $T_n$:
\begin{equation*}
\TT = \ZZ[T_1, T_2, \ldots].
\end{equation*}
A \emph{Hecke eigenform} is a modular form that is an eigenvector of $T_n$ for
every $n$. A Hecke eigenform is \emph{normalised} if $a_1(f) = 1$, where
$a_1(f)$ is as in \eqref{Fourier}. By \cite[Proposition 9.10]{Ste2007}, if $f$
is a normalised Hecke eigenform then
\begin{equation} \label{coincide}
T_n f = a_n(f) f \qquad (n \in \ZZ_{>0}),
\end{equation}
where $a_n(f)$ is as in \eqref{Fourier}. This means that we can compare
Fourier coefficients by studying eigenvalues and eigenspaces of the operators $T_n$.  
A \emph{newform} (in $S_k(\Gamma_0(N))$) is a normalised Hecke eigenform that lies in $S$. 
The proof of \cite[Theorem 5.8.2]{DS2005} shows that the set of newforms in $S$
is a basis for $S$. In particular, the space $S$ contains precisely $\dim S$ newforms.

Let $n \in \ZZ_{>0}$. From
\cite[Theorem 4.5.19]{Miy1997} (see also \cite[Theorem 3.48]{Shi1971}), we see that
the characteristic polynomial $\chi_n$ of $T_n$ acting on $S_k(\Gamma_0(N))$ has
rational integer coefficients. If $f \in S_k(\Gamma_0(N))$ is a normalised
Hecke eigenform then it follows from \eqref{coincide} that
$a_n(f)$ is a root of $\chi_n$, and is therefore an algebraic integer.

To hasten our calculations, we use modular symbols. There is a $\TT$-module
isomorphism
\begin{equation} \label{isoC}
\Phi: S \to \bS_k^{\new}(\Gamma_0(N); \CC)^+
\end{equation}
between $S$ and the plus subspace of the new subspace of the vector
space of cuspidal weight $k$ modular symbols for $\Gamma_0(N)$
over $\CC$ (see \cite[Theorem 8.23]{Ste2007}
and the discussion on \cite[p. 165]{Ste2007}). We perform many
of our calculations in $S^* := \bS_k^\new(\Gamma_0(N); \QQ)^+$. 
The Hecke algebra acts on $S^*$, and there are isomorphisms
\begin{equation} \label{symbols_extend}
S^* \otimes F \simeq \bS_k^\new(\Gamma_0(N); F)^+ \qquad (F= \overline \QQ, \CC)
\end{equation}
of $\TT$-modules. These isomorphisms follow from the definitions in \cite[Ch. 8]{Ste2007}.

Let $B = B(N,k)$ be the Sturm bound, and let 
$f_1, \ldots, f_d$ be the newforms in $S$. There exist
$r_1, \ldots, r_B \in \overline \QQ$ such that
\[
\sum_{n \le B} r_n a_n(f_i) \ne \sum_{n \le B} r_n a_n(f_j)  \qquad (1 \le i < j \le d).
\]
Indeed, if $M$ is a matrix over $\overline \QQ$
with distinct rows then the column span of $M$
contains a vector whose entries are distinct.
(With $P$ a large positive integer, take the first column plus $P$ times
the second column plus $P^2$ times the third column, and so on.)
The linear operator $T = \sum_{n \le B} r_n T_n$ acts irreducibly on $S$,
since its eigenvalues are distinct.
Hence, by \eqref{isoC} and \eqref{symbols_extend}, 
the linear operator $T$ acts irreducibly on $S^* \otimes \CC$,
and therefore uniquely
defines a basis of eigenvectors, up to rescaling.
The space $S^* \otimes \overline\QQ$ is stable under
this action, and $\overline \QQ$ is algebraically closed, so
$S^* \otimes \overline\QQ$ must have a basis $\cB$ of
eigenvectors for $T$. 
By \eqref{isoC} and \eqref{symbols_extend}, the modular symbols
$\Phi(f_1), \ldots, \Phi(f_d)$ form a basis of eigenvectors for the action
of $T$ on $S^* \otimes \CC$.
For $i = 1,2,\ldots,d$, choose
$s_i \in \cB$ equal to a constant times $\Phi(f_i)$. Now 
$f_i \mapsto s_i$ ($1 \le i \le d$) is a $\TT$-module isomorphism
\begin{equation} \label{iso}
S_k^\new(\Gamma_0(N); \overline\QQ) 
\simeq S^* \otimes \overline \QQ.
\end{equation}

\section{The algorithm for computing $n_0(N,k)$}
\label{algorithm}

Let $N, k \in \ZZ_{>0}$, and recall \eqref{Sdef}. We begin with the following observation:

\begin{lemma} \label{primes} Let $f,g \in M_k(\Gamma_0(N))$ be normalised Hecke eigenforms.
Suppose $a_n(f) \ne a_n(g)$ for some $n \in \ZZ_{>0}$. Then there exists a prime 
divisor $p$ of $n$ such that $a_p(f) \ne a_p(g)$.
\end{lemma}

\begin{proof} This follows from \eqref{coincide} and \cite[(5.10)]{DS2005}, upon
noting that $T_{mn} = T_m T_n$ whenever $(m,n) = 1$ (see \cite[\S 5.3]{DS2005}).
\end{proof}

In view of \eqref{coincide}, we now see that if $\dim S \ge 2$ then
$n_0(N,k)$ is the least prime $\ell$ such that
there do not exist distinct newforms $f, g \in S$ such that $f$ and $g$ have the
same $T_p$ eigenvalues for each prime $p \le \ell$.
Our basic algorithm is as follows:

\begin{algorithm} \label{basic}
Build $S$.
\begin{enumerate} 
\item If $\dim S < 2$, return 0.
\item Consider the eigenspaces of the
action of $T_2$ on $S$. Let $A_1, \ldots, A_a$ be the 
eigenspaces of dimension greater than one, and call these
the \emph{homes for $T_2$}. If $a = 0$, return 2.
\item Consider the eigenspaces of the action of $T_3$ on $S$, and intersect
these with $A_1, \ldots, A_a$, separately. Let $B_1, \ldots, B_b$ be the intersections of
dimension greater than one, and call these the \emph{homes for $T_3$}. If $b = 0$, return 3.
\item Repeat for $T_5, T_7, T_{11}, \ldots$. 
\end{enumerate}
\end{algorithm}

As the newforms in $S$ are linearly independent, 
the dimension of any subspace of $S$ is greater than or equal
to the number of newforms it contains. Thus, by the above discussion,
the Sturm bound implies that Algorithm \ref{basic} terminates
and returns an upper bound for $n_0$.
In fact the output of Algorithm \ref{basic} is exactly $n_0$, since
we can show that the dimension of any ``home'' is equal to the number of 
newforms it contains:

\begin{lemma} \label{HomeBasis}
Every home, as defined in steps 2 and 3 of Algorithm \ref{basic},
has a basis given by its newforms.
\end{lemma}

\begin{proof} We induct on primes. As discussed in \S \ref{background}, the space $S$ has a basis given by its newforms. Let $p$ be prime. If $p=2$, let $H=S$. Otherwise, let $H$ be a home for the Hecke operator corresponding to the prime before $p$. Our inductive hypothesis is that the set $\{ f_1, \ldots, f_d \}$ of newforms in $H$ is a basis for $H$. Let $B$ be a home for $T_p$ that comes from intersecting with $H$ in step 3 of Algorithm \ref{basic} (if $p=2$, let $B$ be any home for $T_2$). It remains to show that the newforms in $B$ constitute a basis for $B$.

Note that $T_p$ acts on $H$, since the Hecke operators commute. Further, the home $B$ is an eigenspace of this action. Recalling \eqref{Fourier}, and letting $a_p(f_i) = \lam_i$ ($1 \le i \le d$), we see that the characteristic polynomial of this action is $\prod_{i \le d} (X - \lam_i)$. Let $\lam$ be the eigenvalue associated to $B$, and let $I$ be the set of $i \in \{1,2,\ldots,d\}$ such that $\lam_i = \lam$. The set of newforms in $B$ is $\{f_i: i \in I\}$. This is a basis for $B$, being a linearly independent subset of size $|I| = \dim B$.
\end{proof}

Using the software Sage \cite{Ste2012}, we may implement Algorithm \ref{basic}.
It suffices to work over $\overline{\QQ}$, since the
Fourier coefficients of normalised Hecke eigenforms are algebraic. 
Moreover, by \eqref{iso}, we may use modular symbols instead of modular forms. 
These changes improve the speed of our algorithm. 
They are implemented in Sage as follows:
\begin{algorithm} \label{base_extend}
Build $S^* =\bS_k^{\new}(\Gamma_0(N); \QQ)^+$. Suppose we wish to build the eigenspaces for the action of a Hecke operator on $S_k^\new(\Gamma_0(N); \overline{\QQ})$. Compute the matrix of its action on $S^*$, then use the command \begin{verbatim}base_extend\end{verbatim} to consider it as a matrix with entries in $\overline{\QQ}$. Build the eigenspaces for this matrix. 
\end{algorithm}

We thus produce the eigenspaces in $S^* \otimes \overline \QQ$, which correspond via \eqref{iso} to the eigenspaces in $S_k^\new(\Gamma_0(N); \overline{\QQ})$. The drawback of the algorithm as described thus far (that is, modifying Algorithm \ref{basic} using Algorithm \ref{base_extend}) is that it requires the manipulation of large matrices. To overcome this, we introduce the following refinement:
\begin{itemize}
\item Let $q$ be a prime number.
\item Consider, as a polynomial in $T_2$, the characteristic polynomial of the action of $T_2$ on $S^*$. Factorising this over $\QQ$, consider the irreducible factors of dimension greater than 1, and take their kernels. Call these the \emph{streets for $T_2$}.
\item Compute the corresponding kernels with $T_3$ in place of $T_2$, intersect them with the streets for $T_2$, and take the intersections of dimension greater than one. Call these the \emph{streets for $T_3$}.
\item Repeat for $T_5, \ldots, T_q$.
\item Return the streets for $T_q$, and call these the \emph{final streets}.
\end{itemize}

Let $q$ be prime, let $\cF$ be a final street, 
and let $\cF' = \cF \otimes \overline{\QQ}$. 
We seek to show that running the
algorithm with $\cF$ in place of $S^*$ returns the smallest integer $m$
such that if $f,g \in \cF'$ are newforms such that $a_n(f) = a_n(g)$
for all $n \le m$ then $f=g$ (\emph{newforms} are understood with 
reference to \eqref{iso}). For this purpose, it suffices to obtain the 
appropriate analogue to Lemma \ref{HomeBasis}. By the proof of Lemma \ref{HomeBasis}, 
it remains to show that $\cF'$ has a basis given by its newforms.

\begin{lemma}
Let $q$ be prime, and let $\cF$ be a corresponding final street. Then 
$\cF \otimes \overline{\QQ} $ has a basis given by its newforms.
\end{lemma}

\begin{proof}
By the discussion after Algorithm \ref{base_extend}, we may assume that 
$\fS := S_k^\new(\Gamma_0(N); \overline \QQ)$ 
is used, rather than $S^*$. We regard
streets as intersections of kernels in $\fS$, and 
proceed by induction on primes. 
The base case is $\fS$, which has a basis given by its newforms.
Let $p$ be prime. If $p = 2$, let $F = \fS$. Otherwise, let $F$ be a street
for the Hecke operator corresponding to the prime before $p$. Our inductive
hypothesis is that the set $\{ f_1, \ldots, f_d \}$ of newforms in $F$
is a basis for $F$. 
Consider the characteristic polynomial $\chi_p$ of the action 
of $T_p$ on $\fS$. 
Factorising $\chi_p$ over $\QQ$,
let $P_1, \ldots, P_t$ be the distinct monic irreducible factors, and let 
$X_j = F \cap \ker P_j(T_p)$ ($1 \le j \le t$).
It remains to show that each $X_j$ has a basis given by its newforms.

Recall \eqref{Fourier}, and let $a_p(f_i) = \lam_i$ $(1 \le i \le d)$. 
Fix $j \in \{1,2, \ldots,t\}$,
and let $I$ be the set of $i \in \{1,2,\ldots,d\}$ such that $P_j(\lam_i) = 0$.
The set of newforms in $X_j$ is $\{ f_i : i \in I \}$. This set is linearly
independent, so it remains to show that it spans $X_j$.
Let $f \in X_j$, and write $f = c_1 f_1 + \ldots + c_d f_d$ with
$c_1, \ldots, c_d \in \overline \QQ$. Now
\[
0 = P_j(T_p) (c_1 f_1 + \ldots + c_d f_d) = \sum_{i \le d} c_i P_j(\lam_i) f_i,
\]
so $c_i = 0$ whenever $i \notin I$. Now $f$ lies in the span of 
$\{ f_i : i \in I \}$, which completes the proof.
\end{proof}

Let $q$ be a prime number, and build the final streets as above.
We run our algorithm on each of the final streets,
and consider the maximum output, $m$. Suppose $m \ge q$.
There exist distinct newforms $f,g \in S$ whose 
Fourier coefficients satisfy $a_p(f) = a_p(g)$ for all primes $p < m$; so $n_0 \ge m$.
Further, there cannot exist distinct newforms $f,g \in S$ such that
$a_p(f) = a_p (g)$ for all primes $p \le m$ (otherwise
they would be in $\cF \otimes \overline{\QQ}$ for the same
final street $\cF$, contradicting the fact that
$m$ was the maximum output obtained by running the algorithm on 
the final streets). Hence $n_0 \le m$, so we must have $n_0 = m$.

To summarise the above discussion, if the maximum output is greater
than or equal to our chosen prime number $q$, then it equals $n_0$.
Thus, the following procedure returns $n_0(N,k)$:

\begin{itemize}
\item If $\dim S < 2$, return 0.
\item Choose $q=7$, and run the algorithm on each of the final streets, 
taking the maximum output $m$. If $m \ge q$, return $m$.
\item Repeat for $q=5,3,2$.
\item Return 2.
\end{itemize}

For efficiency, we adopt one final finesse. By similar reasoning to above, 
we note that if $m < q$ then $n_0 \le q$. We can sometimes use this to 
deduce the value of $n_0$ without completing every step of the algorithm.
For instance, if we know that $n_0 \le 7$, and that one of the final
streets for $q=5$ returns 7, then we must have $n_0 = 7$.
Our full Sage \cite{Ste2012} code may be found at the second
author's webpage.\footnote{\url{http://aghitza.org/research/}} A sample of the resulting data
is given in the appendix.

\section{The stability conjecture}
\label{stability}

Our data suggest that if $N$ is fixed then 
$n_0(N,k)$ stabilises as $k$ increases ($k$ even); see
the appendix. The evidence
is particularly compelling when $N$ is squarefree, and we propound a more 
precise statement in this case:

\begin{conjecture}
  \label{conj:stability_squarefree}
  Let $N\in\ZZ_{>0}$ be squarefree. Then there exists $K\in\ZZ_{>0}$ such that if
  $k\geq K$ is an even integer then $n_0(N,k)$ is equal to the least prime that does 
not divide $N$.
\end{conjecture}

Part of this conjecture can be obtained from the following result:
\begin{theorem}[{Atkin-Lehner~\cite[Theorem 3]{AL1970}}]
  \label{thm:atkinlehner}
  Let $N, k\in\ZZ_{>0}$, let $f \in S_k(\Gamma_0(N))$ be a newform,
 and let $p$ be a prime dividing $N$.
  \begin{enumerate}[(a)]
    \item If $p^2\mid N$ then $a_p(f)=0$.
    \item If $p^2\nmid N$ then $a_p(f)=\pm p^{\frac{k}{2}-1}$.
  \end{enumerate}
\end{theorem}

So for primes $p\mid N$, the eigenvalue $a_p(f)$ is heavily prescribed, which
makes the operator $T_p$ particularly bad at telling apart eigenforms. Thus,
we would expect $n_0(N,k)$ to be greater than or equal to the least prime that does 
not divide $N$. This is indeed the case if the space $S = \skn$ is
sufficiently large:
\begin{corollary}
  \label{cor:geq}
  Let $N,k\in\ZZ_{>0}$.  Let $t\in\ZZ_{\geq 0}$ be the number of consecutive
  primes, starting from $2$, that divide $N$ (so $p_{t+1}$ is the least prime that does 
not divide $N$). Suppose
  \begin{equation*}
    \dim S > 2^t.
  \end{equation*}
Then $n_0$ is greater than or equal to the least prime that does 
not divide $N$.
\end{corollary}
\begin{proof}
  Let $i\in\{1,2,\ldots,t\}$, and note that $p_i\mid N$. By
  Theorem~\ref{thm:atkinlehner}, there are at most $2$ possible values for the
  $p_i$-th coefficient of a newform in $S$. (There is one possibility if
  $p_i^2\mid N$, and two possibilities if $p_i^2\nmid N$.)

Since $S$ has a basis given by its newforms, there exist $\dim S\geq 2^t+1$
  distinct newforms in $S$. By the pigeonhole principle, there exist at least
  $2^{t-1}+1$ distinct newforms in $S$ with the same Fourier coefficient $a_2$.  Among
  these, there exist at least $2^{t-2}+1$ distinct newforms with the same
  $a_3$. Continuing in this way, there exist at least $2^{t-t}+1=2$ distinct
  newforms with the same $a_2,a_3,a_5,\ldots,a_{p_t}$. We conclude that $n_0\geq p_{t+1}$.
\end{proof}

Martin~\cite[Theorem 1]{Mar2005} provides a formula for $\dim S$. 
Combining it with Corollary~\ref{cor:geq} gives rise to the 
``easy half'' of Conjecture~\ref{conj:stability_squarefree}:

\begin{theorem}
  \label{thm:stability_easy_half}
  Let $N \in \ZZ_{>0}$, and let $k \ge 38$ be an even integer. Then $n_0(N,k)$ is greater than or equal to the least prime that does not divide $N$.
\end{theorem}

\begin{proof} As $k > 2$ is even, Martin~\cite[Theorem 1]{Mar2005} gives
\begin{equation} \label{MartinFormula}
\dim S = (k-1)N s_0^+(N)/12 - v_\infty^+(N)/2 + c_2(k) v_2^+(N) + c_3(k) v_3^+(N),
\end{equation}
where $s_0^+$, $v_\infty^+$, $c_2$, $v_2^+$, $c_3$ 
and $v_3^+$ are certain quickly computable arithmetic functions. 

First suppose $N \ge 1000$. We deduce from \eqref{MartinFormula}
and the definitions of the arithmetic functions therein that 
\[
\dim S \ge
 (k-1)N/12 \times \Bigl( \prod_{p | N} (1- p^{-1}- p^{-2}) \Bigr)
- \sqrt{N}/2 - 17/12 \times 2^{\omega(N)}.
\]
Since $t \le \omega(N)$ in Corollary \ref{cor:geq}, 
it now remains to show that 
\begin{equation} \label{goal}
k > 1 + 12/N \times (\sqrt{N}/2 + 29/12 \times 2^{\omega(N)}) 
\times \prod_{p | N} (1-p^{-1}- p^{-2})^{-1}.
\end{equation}
We may easily verify the bounds
\[
\prod_{p | N} (1-p^{-1}- p^{-2})^{-1} \le 20/9 \times (9/5)^{\omega(N)},
\]
$(9/5)^{\omega(N)} < 2.8N^{1/4}$ and $2^{\omega(N)} < 5N^{1/4}$. Since
\[
k \ge 38 > 1 + 6.23(6N^{-1/4} + 145 N^{-1/2})
\]
for all $N \ge 1000$, we now have \eqref{goal}.

For $N = 2, 4, 6, 12$, we shall use \eqref{MartinFormula} to check
that the hypothesis of Corollary \ref{cor:geq} is met.
We note that $-3/4 < c_2(k) \le 1/4$ and $-2/3 < c_3(k) \le 1/3$.
If $N=2$ then \eqref{MartinFormula} yields
\[
\dim S = (k-1)/12 - c_2(k) - 2c_3(k) \ge 37/12 - 1/4 - 2/3 > 2 = 2^t.
\]
If $N = 4$ then
\eqref{MartinFormula} yields
\[
\dim S = (k-1)/12 - c_2(k) + c_3(k) > 37/12 - 1/4 - 2/3 > 2 = 2^t.
\]
If $N = 12$ then
\eqref{MartinFormula} yields
\[
\dim S = (k-1)/6 + 2c_2(k) - c_3(k) > 37/6 - 3/2 - 1/3 > 4 = 2^t.
\]
Consider $N=6$. By \eqref{MartinFormula},
it suffices to prove that
\[
(k-1)/6 + 2c_2(k) + 2c_3(k) > 4.
\]
We can verify this directly for $k= 38, 40$, while if $k \ge 42$ then
\[
(k-1)/6 + 2c_2(k) + 2c_3(k) > 41/6 - 3/2 - 4/3 = 4.
\]

Finally, suppose that $N < 1000$ with $N \notin \{ 2,4,6,12 \}$. By direct computation, we have 
\[ 
k \ge 38 > 1+\frac{12 (v_\infty^+(N)/2  + 3| v_2^+(N)|/4 + 2|v_3^+(N)|/3 + 2^t)  }
{N s_0^+(N)}.
\]
The proof is completed via \eqref{MartinFormula}
and Corollary \ref{cor:geq},
recalling that $|c_2(k)| < 3/4$ and $|c_3(k)| < 2/3$.
\end{proof}

The following example suggests that the non-squarefree level case is more complicated. 
We take the following definition from \cite[p. 2]{Tsa2014}. Let $N, k \in \ZZ_{>0}$, 
let $f \in S_k(\Gamma_0(N))$, and let $\chi$ be a Dirichlet character.
The \emph{twist of $f$ by $\chi$} is given by
\[
f \otimes \chi (q) = \sum_{n=1}^\infty \chi(n) a_n(f) q^n.
\]

\begin{example}
  Let $S = S_k^{\mathrm{new}}(\Gamma_0(49))$, and let $\chi$ be the Legendre symbol modulo 7, i.e.
\[ 
\chi(n) = \Big(\frac n7 \Big) \qquad (n \in \ZZ).
\]
For each $k \in \{4,6,8,10,12\}$, we observe newforms $f,g \in S$ such that 
$f = g \otimes \chi$ and $g = f \otimes \chi$. As $\chi(2) = 1$, 
we have $a_2(f) = a_2(g)$, so $T_2$ fails to distinguish $f$ and $g$. 
\end{example}

This phenomenon is closely related to the existence of CM forms 
(see \cite[p. 2]{Tsa2014}). 
Indeed, in the example above $f+g$ has complex multiplication by $\chi$.  It
is likely that the forms $f$ and $g$, although ``new'' in the usual sense
(not arising from $\Gamma_0(7)$ or $\Gamma_0(1)$), are coming from a different
congruence subgroup $\Gamma$ of level $7$.  For an explanation of this type of
behaviour for $\Gamma_0(9)$, see~\cite{Str2012}.

There may be a more general stability phenomenon which also encompasses 
non-squarefree values of $N$.
In the cases $N=49, 108, 147, 225$, one might
predict from the data that $n_0$ stabilises towards a prime that exceeds the 
least prime that does not divide $N$
(see Tables \ref{table5} and \ref{table8}).
For all other values of $N$ that we examined, it would appear that $n_0$
stabilises towards the least prime that does not divide $N$.
Does there always exist $K\in\ZZ_{>0}$ 
such that if $k\geq K$ is an even integer then $n_0(N,k)=n_0(N,K)$?

\section{Irreducibility of Hecke polynomials}
\label{gen_maeda}

A \emph{Hecke polynomial} is the characteristic polynomial of a Hecke operator
$T_n$ acting on a space of modular forms. In the 1970s, Maeda observed that the Hecke polynomials of $T_2$ on
$S_k(\SL_2(\ZZ))$ are irreducible over $\QQ$ for all $k$ such that $\dim
S_k(\SL_2(\ZZ)) \leq 12$.  Over the next 20 years, this observation matured
into the following statement:

\begin{conjecture}[Maeda~\cite{HM1997}]
  \label{conj:maeda}
  Let $k\in\ZZ_{>0}$, $n\in\ZZ_{>1}$, and let $F\in\ZZ[X]$ be the
  characteristic polynomial of the Hecke operator $T_n$ acting on
  $\cS:=S_k(\SL_2(\ZZ))$.  Then $F$ is irreducible over $\QQ$ and its Galois
  group $G$ is isomorphic to the symmetric group $\Sigma_{\dim \cS}$.
\end{conjecture}

We refer the reader to~\cite{GM2012} for a survey of results on Maeda's
conjecture and a report on its verification for the operator $T_2$ and weights
$k\leq\numprint{14000}$.

In~\cite{Tsa2014}, Tsaknias considers higher-level generalisations of the
following weak version of Conjecture~\ref{conj:maeda}: \emph{there is a unique
Galois orbit of Hecke eigenforms in $S_k(\SL_2(\ZZ))$}.  We describe
his findings in the squarefree level case.  If $N=p_1p_2\ldots p_t$ is
squarefree, the Atkin-Lehner involutions $w_{p_i}$ decompose the space of
newforms into eigenspaces
\begin{equation*}
  \skn = \bigoplus_{\epsilon\in\{\pm 1\}^t} S_\epsilon.
\end{equation*}
Tsaknias's computations indicate that, for $k$ large enough, the space
$\skn$ has $2^t$ Galois orbits of newforms. 
The second author and J. Withers~\cite{GW2015} have
investigated higher-level analogues of the full Maeda conjecture. Their
experiments suggest the following statement in the squarefree level case:

\begin{conjecture}
  \label{conj:gen_maeda}
  Let $N\in\ZZ_{>0}$ be squarefree. Then there exists $K^\prime\in\ZZ_{>0}$
  such that the following hold whenever $k \ge K^\prime$ is even and $n \ge 2$
  is coprime to $N$:
\begin{enumerate}[(a)]
    \item The characteristic polynomial $F$ of the Hecke operator $T_n$ acting
      on $\skn$ is separable (that is, $F$ has no repeated roots over
      $\overline{\QQ}$).
    \item The Atkin-Lehner decomposition
      \begin{equation*}
        \skn = \bigoplus_{\epsilon} S_\epsilon
      \end{equation*}
      is the only obstacle to the irreducibility of the polynomial $F$.
  \end{enumerate}
\end{conjecture}

These statements have been verified computationally for squarefree $N\leq
200$, even weights $k\leq 30$ and operators $T_p$ for $p <100$ prime and not
dividing $N$. Our immediate interest in Conjecture~\ref{conj:gen_maeda} is the following
result:

\begin{theorem}
  Part (a) of the generalised Maeda Conjecture~\ref{conj:gen_maeda} implies the stability
  Conjecture~\ref{conj:stability_squarefree}.
\end{theorem}
\begin{proof}
  Let $N\in\ZZ_{>0}$ be squarefree. Let $K=\max\{38,K^\prime\}$, with
  $K^\prime$ provided by
  Conjecture~\ref{conj:gen_maeda}. Let $p$ be the least prime that does 
not divide $N$, and let $k\geq K$ be even. By
 Theorem~\ref{thm:stability_easy_half}, we have $n_0(N,k)\geq p$. So it
  suffices to show that if $f,g \in S_k(\Gamma_0(N))$ are distinct
newforms then $a_p(f)\neq a_p(g)$.

  Let $\mathcal{B}=\{f_1,f_2,\ldots,f_d\}$ be the newforms in
  $S_k(\Gamma_0(N))$, where $d=\dim\skn$. We know that $\mathcal{B}$ is a
  basis for $\skn$, so $\{a_p(f_1), a_p(f_2),\ldots, a_p(f_d)\}$ is precisely
  the set of roots of the characteristic polynomial of $T_p$.
  But by part (a) of Conjecture~\ref{conj:gen_maeda}, this polynomial has $d$
  distinct roots, hence $a_p(f_i)\neq a_p(f_j)$ for $i\neq j$.
\end{proof}

\section*{Appendix: Data}

We tabulate $n_0 = n_0(N,k)$ for $k$ even (if $k$ is odd then $n_0 = 0$).
We put $N$ on the vertical axis and $k$ on the horizontal axis, so that
along any row $N$ is fixed and $k$ varies. A larger set of 
data may be found at the second author's
webpage.\footnote{\url{http://aghitza.org/research/}}

\FloatBarrier
\begin{table}[h!]
\caption{Some values of $n_0(N,k)$.}
\label{table1}
\begin{center}
\begin{tabular*}{\textwidth}{@{} |r| @{\extracolsep{\fill} } cccccccccc|}
\hline
\diaghead{\theadfont[width=11em]}
{N}{k}&
38&40&42&44&46&48&50&52 & 54&56\\
\hline
1 &2&2&2&2&2&2&2&2&2&2\\
2 &3&3&3&3&3&3&3&3&3 & 3\\
3 &2&2&2&2&2&2&2&2&2&2\\
4 &3&3&3&3&3&3&3&3&3&3\\
5 &2&2&2&2&2&2&2&2&2&2\\
6 &5&5&5&5&5&5&5&5&5&5\\
\hline
\end{tabular*}
\end{center}
\end{table}
\FloatBarrier

\FloatBarrier
\begin{table}[h!]
\caption{More values of $n_0(N,k)$.}
\begin{center}
\begin{tabular*}{\textwidth}{@{} |r| @{\extracolsep{\fill} } cccccccccc|}
\hline
\diaghead{\theadfont[width=11em]}
{N}{k}&
24&26&28&30&32&34&36&38&40&42\\
\hline
7 &2&2&2&2&2&2&2&2&2&2\\
8 &3&3&3&3&3&3&3&3&3&3\\
9 &2&2&2&2&2&2&2&2&2&2\\
10 &3&3&3&3&3&3&3&3&3&3\\
11 &2&2&2&2&2&2&2&2&2&2\\
12 &5&5&5&5&5&5&5&5&5&5\\
\hline
\end{tabular*}
\end{center}
\end{table}
\FloatBarrier

\FloatBarrier
\begin{table}[h!]
\caption{More values of $n_0(N,k)$.}
\begin{center}
\begin{tabular*}{\textwidth}{@{} |r| @{\extracolsep{\fill} } cccccccccc|}
\hline
\diaghead{\theadfont[width=11em]}
{N}{k}&
12&14&16&18&20&22&24&26&28&30\\
\hline
13&2&2&2&2&2&2&2&2&2&2\\
14&3&3&3&3&3&3&3&3&3&3\\
15&2&2&2&2&2&2&2&2&2&2\\
16&3&3&3&3&3&3&3&3&3&3\\
17&2&2&2&2&2&2&2&2&2&2\\
18&5&5&5&5&5&5&5&5&5&5\\
19&2&2&2&2&2&2&2&2&2&2\\
20&3&3&3&3&3&3&3&3&3&3\\
21&2&2&2&2&2&2&2&2&2&2\\
22&3&3&3&3&3&3&3&3&3&3\\
23&2&2&2&2&2&2&2&2&2&2\\
24&5&5&5&5&5&5&5&5&5&5\\
25&2&2&2&2&2&2&2&2&2&2\\
26&3&3&3&3&3&3&3&3&3&3\\
\hline
\end{tabular*}
\end{center}
\end{table}
\FloatBarrier

\FloatBarrier
\begin{table}[h!]
\caption{More values of $n_0(N,k)$.}
\begin{center}
\begin{tabular*}{\textwidth}{@{} |r| @{\extracolsep{\fill} } cccccccccc|}
\hline
\diaghead{\theadfont[width=11em]}
{N}{k}&
10&12&14&16&18&20&22&24&26&28\\
\hline
27&2&2&2&2&2&2&2&2&2&2\\
28&3&3&3&3&3&3&3&3&3&3\\
29&2&2&2&2&2&2&2&2&2&2\\
30&5&5&7&7&7&7&7&7&7&7\\
31&2&2&2&2&2&2&2&2&2&2\\
32&2&3&3&3&3&3&3&3&3&3\\
33&2&2&2&2&2&2&2&2&2&2\\
34&3&3&3&3&3&3&3&3&3&3\\
35&2&2&2&2&2&2&2&2&2&2\\
36&5&5&5&5&5&5&5&5&5&5\\
37&2&2&2&2&2&2&2&2&2&2\\
\hline
\end{tabular*}
\end{center}
\end{table}
\FloatBarrier

\FloatBarrier
\begin{table}[h!]
\caption{More values of $n_0(N,k)$.}
\label{table5}
\begin{center}
\begin{tabular*}{\textwidth}{@{} |r| @{\extracolsep{\fill} } cccccccccccc|}
\hline
\diaghead{\theadfont[width=11em]}
{N}{k}&
2&4&6&8&10&12&14&16&18&20&22&24 \\
\hline
38&2&3&3&3&3&3&3&3&3&3&3&3\\
39&2&2&2&2&2&2&2&2&2&2&2&2\\
40&0&3&3&3&3&7&3&3&3&3&3&3\\
41&2&2&2&2&2&2&2&2&2&2&2&2\\
42&0&3&5&5&5&5&5&5&5&5&5&5\\
43&2&2&2&2&2&2&2&2&2&2&2&2\\
44&0&3&3&3&3&3&3&3&3&3&3&3\\
45&0&2&2&2&2&2&2&2&2&2&2&2\\
46&0&3&3&3&3&3&3&3&3&3&3&3\\
47&2&2&2&2&2&2&2&2&2&2&2&2\\
48&0&5&5&5&5&5&5&5&5&5&5&5\\
\hline
49&0&3&3&3&3&3&3&3&3&3&3&3\\
\hline
50&2&3&3&3&3&3&3&3&3&3&3&3\\
51&2&3&2&2&2&2&2&2&2&2&2&2\\
52&0&3&3&3&3&3&3&3&3&3&3&3\\
53&2&2&2&2&2&2&2&2&2&2&2&2\\
54&2&5&5&5&5&5&5&5&5&5&5&5\\
55&2&2&2&2&2&2&2&2&2&2&2&2\\
56&3&3&3&3&3&3&3&3&3&3&3&3\\
57&3&2&2&2&2&2&2&2&2&2&2&2\\
\hline
\end{tabular*}
\end{center}
\end{table}
\FloatBarrier

\FloatBarrier
\begin{table}[h!]
\caption{More values of $n_0(N,k)$.}
\begin{center}
\begin{tabular*}{\textwidth}{@{} |r| @{\extracolsep{\fill} } cccccccccccc|}
\hline
\diaghead{\theadfont[width=11em]}
{N}{k}&
2&4&6&8&10&12&14&16&18&20&22&24 \\
\hline
58&2&3&3&3&3&3&3&3&3&3&3&3\\
59&2&2&2&2&2&2&2&2&2&2&2&2\\
60&0&5&5&5&7&7&7&7&7&7&7&7\\
61&2&2&2&2&2&2&2&2&2&2&2&2\\
62&3&3&3&3&3&3&3&3&3&3&3&3\\
63&2&2&2&2&2&2&2&2&2&2&2&2\\
64&0&3&3&3&3&3&3&3&3&3&3&3\\
65&2&2&2&2&2&2&2&2&2&2&2&2\\
66&3&5&7&5&5&5&5&5&5&5&5&5\\
67&2&2&2&2&2&2&2&2&2&2&2&2\\
68&3&3&3&3&3&3&3&3&3&3&3&3\\
\hline
\end{tabular*}
\end{center}
\end{table}
\FloatBarrier

\FloatBarrier
\begin{table}[h!]
\caption{More values of $n_0(N,k)$.}
\begin{center}
\begin{tabular*}{\textwidth}{@{} |r| @{\extracolsep{\fill} } cccccccccc|}
\hline
\diaghead{\theadfont[width=11em]}
{N}{k}&
2&4&6&8&10&12&14&16&18&20 \\
\hline
69&2&2&2&2&2&2&2&2&2&2\\
70&0&3&3&3&3&3&3&3&3&3\\
71&2&2&2&2&2&2&2&2&2&2\\
72&0&5&5&5&5&5&5&5&5&5\\
73&2&2&2&2&2&2&2&2&2&2\\
74&3&3&3&3&3&3&3&3&3&3\\
75&2&2&2&2&2&2&2&2&2&2\\
76&0&3&3&3&3&3&3&3&3&3\\
77&3&2&2&2&2&2&2&2&2&2\\
78&0&5&5&5&5&5&5&5&5&5\\
79&2&2&2&2&2&2&2&2&2&2\\
80&3&7&3&3&3&7&3&3&3&3\\
81&2&2&2&2&2&2&2&2&2&2\\
82&3&3&3&3&3&3&3&3&3&3\\
83&2&2&2&2&2&2&2&2&2&2\\
84&3&3&5&5&5&5&5&5&5&5\\
85&2&3&2&2&2&2&2&2&2&2\\
86&3&3&3&3&3&3&3&3&3&3\\
87&2&2&2&2&2&2&2&2&2&2\\
\hline
\end{tabular*}
\end{center}
\end{table}
\FloatBarrier

\FloatBarrier
\begin{table}[h!]
\caption{More values of $n_0(N,k)$.}
\label{table8}
\begin{center}
\begin{tabular*}{\textwidth}{@{} |r| @{\extracolsep{\fill} } cccccccc|}
\hline
\diaghead{\theadfont[width=11em]}
{N}{k}&
2&4&6&8&10&12&14&16\\
\hline
88&3&3&3&3&3&3&3&3\\
89&2&2&2&2&2&2&2&2\\
90&5&7&7&7&7&7&7&7\\
91&2&2&2&2&2&2&2&2\\
92&3&3&3&3&3&3&3&3\\
93&2&2&2&2&2&2&2&2\\
94&3&3&3&3&3&3&3&3\\
95&2&3&2&2&2&2&2&2\\
96&3&5&5&5&5&5&5&5\\
97&2&2&2&2&2&2&2&2\\
98&3&3&3&3&3&3&3&3\\
99&5&2&2&2&2&2&2&2\\
100&0&3&3&3&3&3&3&3\\
\hline
108&0&7&7&5&7&5&5&7\\
147&3&5&3&3&3&3&&\\
225&7&7&7&7&7&7&&\\
\hline
\end{tabular*}
\end{center}
\end{table}
\FloatBarrier

\end{document}